\documentclass[12pt]{amsart}
\usepackage{amsmath,amssymb}
\usepackage{amsfonts}
\usepackage{amsthm}
\usepackage{latexsym}
\usepackage{graphicx}

\def\lf{\left}
\def\ri{\right}

\def\id{\mbox{id}}

\def\p{\partial}

\def\R{\mathbb{R}}

\def\vv<#1>{\langle#1\rangle}

\def\XXint#1#2{\setbox0=\hbox{$#1{#2}{\int}$}{#2}\kern-.5\wd0 }

\def\XXint#1#2#3{{\setbox0=\hbox{$#1{#2#3}{\int}$}
     \vcenter{\hbox{$#2#3$}}\kern-.5\wd0}}

\def\vv<#1>{\langle#1\rangle}
\def\pt{\frac{\p}{\p t}}
\def\pu{\frac{\p}{\p u}}

\newtheorem{thm}{Theorem}[section]

\newtheorem{lem}{Lemma}[section]

\theoremstyle{definition}
\newtheorem{defn}{Definition}[section]
\theoremstyle{remark}

\numberwithin{equation}{section}
\def\dev{{\rm dev}}

\begin{document}

\title{de Rham decomposition for Riemannian manifolds with boundary}

\author{Chengjie Yu$^1$}
\address{Department of Mathematics, Shantou University, Shantou, Guangdong, 515063, China}
\email{cjyu@stu.edu.cn}

\thanks{$^1$Research partially supported by GDNSF with contract no. 2021A1515010264 and NNSF of China with contract no. 11571215.}

\renewcommand{\subjclassname}{%
  \textup{2010} Mathematics Subject Classification}
\date{}
\begin{abstract}
In this paper, we extend the classical de Rham decomposition theorem to the case of Riemannian manifolds with boundary by using the trick of development of curves.
\end{abstract}
\maketitle\markboth{Chengjie Yu }{de Rham decomposition}
\section{Introduction}
Let $(M^n,g)$ be a simply connected complete Riemannian manifold with two nontrivial parallel distributions $T_1$ and $T_2$ that are orthogonal complement of each other. Then, $M$ is isometric to a Riemnnian product $M_1\times M_2$ with $M_i$ a maximal integral submanifold of $T_i$ for $i=1,2$.
This is a classical result in differential geometry obtained by de Rham \cite{dR} in 1952. In 1962, Wu \cite{Wu} extended the result to simply connected complete semi-Riemannian manifolds. The strategy of de Rham's proof is to patch up local product decompositions to a global one. This strategy was taken up and presented in a modern form by Maltz \cite{Ma} using an idea for patching up local isometries by O'Neil \cite{Ne}. Wu's strategy of proof is different. He used the theorem of Cartan-Ambrose-Hicks to construct a global isometry from $M_1\times M_2$ to $M$.  In fact, Maltz \cite{Ma} extended de Rham's decomposition Theorem to complete affine manifolds. The de Rham decomposition Theorem was also extended to non-simply connected manifolds by Eschenburg-Heintz \cite{EH} and to geodesic spaces by Foertsch-Lytchak \cite{FL}. The uniqueness of Wu's de Rham decomposition for indefinite metrics was just shown by Chen \cite{Ch} recently.

In this paper, we extend de Rham's decomposition theorem to complete Riemannian manifolds with boundary. Here, for completeness of a Riemannian manifold we mean metric completeness. Note that if a Riemannian manifold with boundary is decomposable, then it must be decomposed as a product of a Riemannian manifold with boundary and a Riemannian manifold without boundary because of the smoothness of the boundary. This implies that the outward normal on the boundary must be all contained in one of the two parallel distributions. Our result confirms the converse of the above observation.
\begin{thm}\label{thm-main}
Let $(M^n,g)$ be a simply connected complete Riemannian manifold with boundary. Let $T_1$ and $T_2$ be two nontrivial parallel distributions that are orthogonal complements of each other. Suppose that $T_1$ contains the normal vectors on $\p M$. Let $p$ be an interior point of $M$ and $\iota_i:(M_i,p_i)\hookrightarrow (M,p)$ be the simply connected leaf of the foliation $T_i$ passing through $p$ for $i=1,2$. Then, $M_1$ is a manifold with boundary and $M_2$ is a manifold without boundary, and moreover, there is an isometry  $f:M_1\times M_2\to M$ such that $f(p_1,p_2)=p$ and $f_{*(p_1,p_2)}={\iota_1}_{*p_1}+{\iota_2}_{*p_2}$.
\end{thm}
We would like to mention that the assumption on the simply connectedness of $M$ can not be removed. For example, let $M=[0,1]\times \R^2/\mathbb{Z}^2$ equipped with the standard product metric, and
\begin{equation}
T_1=\mbox{span}\left\{\pt, \frac{\p}{\p x}+r\frac{\p}{\p y}\right\}
\end{equation}
and
\begin{equation}
T_2=\mbox{span}\left\{-r\frac{\p}{\p x}+\frac{\p}{\p y}\right\}
\end{equation}
with $r$ an irrational number, where $t$ is the natural coordinate on $[0,1]$ and $(x,y)$ is the natural coordinate on $\R^2$. Then, $T_1$ and $T_2$ are parallel distributions on $M$ that are orthogonal complements of each other with $T_1$ containing the normal vectors. However, we can not have a decomposition of $M$ according the distributions  $T_1$ and $T_2$ because $r$ is an irrational number.

Because Wu's proof used geodesics to connect two different points and Maltz's proof relied heavily on convex normal neighborhoods, their proofs will not work for Riemnnian manifolds with boundary without any convexity assumption on the boundary. We will prove the result by combining the idea of Kobayashi-Nomizu \cite[P.187]{KN} using development of curves and the idea of Wu using the Cartan-Ambrose-Hicks theorem.

Let's recall the notion of developments of curves in \cite[P. 130]{KN}.  The original definition in \cite{KN} was given in the language of connections for principle bundles. We will present here an equivalent notion in a more elementary form.
\begin{defn}
Let $(M^n,g)$ be a Riemanian manifold and $v:[0,T]\to T_pM$ be a curve in $T_pM$. A curve $\gamma:[0,T]\to M$ such that
$$\gamma(0)=p\ \mbox{and}\  \gamma'(t)=P_0^t(\gamma)(v(t))\ \mbox{for any}\ t\in [0,T]$$ is called a development of the curve $v$. Here $P_{t_1}^{t_2}(\gamma):T_{\gamma(t_1)}M\to T_{\gamma(t_2)}M$ means the parallel displacement along $\gamma$ from $\gamma(t_1)$ to $\gamma(t_2)$.
\end{defn}
Note that when $v$ is constant, the development of $v$ is just a geodesic, and when $v$ is piece-wise constant, the development of $v$ is just a broken geodesic. It can be shown that when $v$ is smooth, the development of $v$ is unique if exists. When the Riemannian manifold is complete, the development of $v$ exists for any $v$.

It is clear that local isometry of Riemannian manifolds will preserve curvature tensors. It was Cartan \cite{Ca} first gave a converse of this fact in local settings. This result is nowadays called Cartan's lemma. The conclusion was extended to a global setting by Ambrose \cite{Am} under the assumptions of simply connectedness and that curvature tensors are preserved by parallel displacements along broken geodesics. Finally, Hicks \cite{Hi} extended the conclusion to complete affine manifolds. A more general form of the Cartan-Ambrose-Hicks theorem can be found in \cite{BH}. In \cite{Ne}, O'Neil gave an alternative proof of Ambrose's result.

In this paper, to implement the idea of Wu proving de Rham decomposition using the Cartan-Ambrose-Hicks theorem, we need the following version of  Cartan-Ambrose-Hicks theorem.
\begin{thm}\label{thm-CAH}
Let $(M^n,g)$ and $(\tilde M^n,\tilde g)$ be two Rimannian manifolds (not necessary complete and may have boundaries). Let $p$ be an interior point of $M$, $\tilde p\in \tilde M$ and $\varphi:T_pM\to T_{\tilde p}\tilde M$ be a linear isometry. Suppose that $M$ is simply connected and for any smooth interior curve $\gamma:[0,1]\to M$ with $\gamma(0)=p$, the development $\tilde \gamma$ of $\varphi(v_\gamma)$ exists in $\tilde M$. Here $$v_\gamma(t)=P_t^0(\gamma)(\gamma'(t))$$ for $t\in [0,1]$. Moreover, suppose that $$\tau^{*}_\gamma R_{\tilde M}=R_M$$ for any smooth interior curve $\gamma:[0,1]\to M$ with $\gamma(0)=p$ where $$\tau_\gamma=P_0^1(\tilde \gamma)\circ\varphi\circ P_1^0(\gamma): T_{\gamma(1)}M\to T_{\tilde\gamma(1)}\tilde M.$$
Here a curve $\gamma:[0,1]\to M$ is said be an interior curve if $\gamma(t)$ is in the interior of $M$ for any $t\in [0,1)$, and $R_M$ and $R_{\tilde M}$ are the curvature tensors of $M$ and $\tilde M$ respectively. Then, the map $f(\gamma(1))=\tilde\gamma(1)$ from $M$ to $\tilde M$ is well defined and $f$ is the local isometry from $M$ to $\tilde M$ with $f(p)=\tilde p$ and $f_{*p}=\varphi$.
\end{thm}
 Our proof of Theorem \ref{thm-CAH} is similar to the proof of Cartan's lemma using the Jacobi field equation. Because we are considering variations for developments of curves, we need the equations of the variation fields for variations of developments of curves that may be considered as a generalization of the equation for Jacobi fields. Here we require the curve $\gamma$ to be interior in Theorem \ref{thm-CAH} because of a technical reason for its application in proving Theorem \ref{thm-main}. One can see from the proof of Theorem \ref{thm-CAH} that the conclusion of Theorem \ref{thm-CAH} is still true if the assumption that $\gamma$ is interior is removed .

 We would like to mention that  by using the trick developed in this paper, we are able to  obtain a decomposition result in \cite{SY2} when replacing the assumption of simply connectedness of the manifold by simply connectedness of one of the factors. We have also used this trick to extend the fundamental theorem for submanifolds to general ambient spaces in \cite{Yu1}. Note the the product of two manifolds with boundary is not a manifold with boundary. It is a manifold with corners (see \cite{JO}). So, there is an interesting question if one can have a more general de Rham decomposition theorem for Riemannian manifolds with corners. The argument in this paper may help to solve the problem. However, because our argument relies heavily on the smoothness of the boundary (see the proof of Lemma \ref{lem-2}), our proof does work for the case of Riemannian manifolds with corners.

The rest of the paper is organized as follows. In Section 2, we prove the local existence and uniqueness for developments, and prove Theorem \ref{thm-CAH}. In Section 3, we prove Theorem \ref{thm-main}.
\section{Developments of curves and Cartan-Ambrose-Hicks Theorem}
In this section, we first give some preliminaries on developments of curves. For completeness of the paper, we also give a proof of the local existence and uniqueness for development of curves. Then, we prove  theorem \ref{thm-CAH}.

\begin{lem}Let $(M^n,g)$ be a Riemannian manifold (with or without boundary) and $p$ be an interior point of $M$. Let $v:[0,T]\to T_pM$ be a smooth curve in $T_pM$. Then, there is a positive number $\epsilon$ and a unique smooth curve $\gamma:[0,\epsilon]\to M$ such that $\gamma(0)=p$ and
\begin{equation}
\gamma'(t)=P_0^t(\gamma)(v(t))
\end{equation}
for $t\in [0,\epsilon]$.
\end{lem}
\begin{proof} We only need to derive the equation of $\gamma$. The conclusion will follow directly by existence and uniqueness of solution for Cauchy problems of ordinary differential equations.

Let $(x^1,x^2,\cdots,x^n)$ be a local coordinate at $p$ with $x^i(p)=0$ for $i=1,2,\cdots,n$. Suppose that
\begin{equation}
v(t)=v^i(t)\frac{\p}{\p x^i}\bigg|_p.
\end{equation}
Let $\gamma(t)=(x^1(t),x^2(t),\cdots,x^n(t))$ be a development of $v$ and $E_i$ be the parallel extension of  $\frac{\p}{\p x^i}\big|_p$ along $\gamma$. Suppose that
\begin{equation}
E_i(t)=x_i^j(t)\frac{\p}{\p x^j}.
\end{equation}
Then,
\begin{equation}\label{eq-x-ij}
\frac{d x_i^j}{dt}+x_i^k\frac{dx^l}{dt}\Gamma_{kl}^j(x^1,x^2,\cdots,x^n)=0
\end{equation}
for any $i,j=1,2,\cdots,n$, since $\nabla_{\gamma'}E_i=0$.

Moreover, note that
\begin{equation}
P_0^t(\gamma)(v(t))=v^i(t)E_i=v^ix_i^j\frac{\p}{\p x^j}.
\end{equation}
So,
\begin{equation}\label{eq-x-i}
\frac{dx^i}{dt}=v^jx_j^i
\end{equation}
for $i=1,2,\cdots,n$, by that  $\gamma'(t)=P_0^t(\gamma)(v(t))$.

In summary, by substituting \eqref{eq-x-i} into \eqref{eq-x-ij}, we know that the curve $\gamma$ must satisfy the following ODEs:
\begin{equation}\label{eq-dev}
\left\{\begin{array}{ll}\frac{dx^i}{dt}=v^jx_j^i&\mbox{for $i=1,2,\cdots,n$}\\
\frac{d x_i^j}{dt}+v^mx_i^kx_m^l\Gamma_{kl}^j(x^1,x^2,\cdots,x^n)=0&\mbox{for $i,j=1,2,\cdots,n$}\\
x_i^j(0)=\delta_i^j&\mbox{for $i,j=1,2,\cdots,n$}\\
x_i(0)=0&\mbox{for $i=1,2,\cdots,n$.}
\end{array}\right.
\end{equation}
By standard theory of ODE, the equation has a unique solution for a short time. This completes the proof of the lemma.
\end{proof}

By combining the local uniqueness and a standard trick in extending solutions for ODEs, one has the following global existence and uniqueness of development of curves for complete Riemannian manifolds without boundary. One can find the proof in \cite[P. 175]{KN}.
\begin{thm}
Let $(M^n,g)$ be a complete Riemannian manifold without boundary. Then, each smooth curve $v:[0,T]\to T_pM$ has a unique development $\gamma:[0,T]\to M$.
\end{thm}

We will denote the development of $v$ as $\dev(p,v)$. When $v$ is constant, it is clear that
\begin{equation}
\dev(p,v)(t)=\exp_p(tv).
\end{equation}
Moreover, it is clear that
\begin{equation}\label{eq-dev-group}
\dev(p,v)(t)=\dev(\dev(p,v)(t_0),P_0^{t_0}(\dev(p,v))v_{t_0})(t-t_0)
\end{equation}
for any $t_0\in [0,T]$ and $t\in [t_0,T]$. Here
$$v_{t_0}(t)=v(t_0+t)$$ for $t\in [0,T-t_0]$. For simplicity, we will denote a vector and its parallel displacement by the same symbol when it makes no confusions. Under this convention, the identity \eqref{eq-dev-group} can be simply written as
\begin{equation}\label{eq-dev-group-1}
\dev(p,v)(t)=\dev(\dev(p,v)(t_0),v_{t_0})(t-t_0).
\end{equation}

Next, we come to derive the equation for the variation field of a variation for developments of curves which can be viewed as a generalization of the Jacobi field equation.
\begin{lem}\label{lem-g-Jacobi}
Let $(M^n,g)$ be a Riemannian manifold and $p\in M$. Let $v(u,t):[0,1]\times [0,1]\to T_pM$ be a smooth map and $$\Phi(u,t)=\dev(p,v(u,\cdot))(t).$$
Let $e_1,e_2,\cdots,e_n$ be an orthonormal basis of  $T_pM$ and $E_i(u,t)$ be the parallel translation of $e_i$ along $\Phi(u,\cdot)$ for $i=1,2,\cdots,n$. Suppose that
\begin{equation}
v(u,t)=\sum_{i=1}^nv_i(u,t)e_i
\end{equation}
and
\begin{equation}\label{eq-pu}
\pu:=\frac{\partial \Phi}{\p u}=\sum_{i=1}^n U_iE_i.
\end{equation}
Moreove, suppose that
\begin{equation}\label{eq-pu-E}
\nabla_{\pu}E_i(u,t)=\sum_{j=1}^nX_{ij}E_j.
\end{equation}
Then,
\begin{equation}\label{eq-g-Jacobi}
\begin{split}
\left\{\begin{array}{ll}U''_i=\sum_{j,k,l=1}^nv_kv_lR(E_k,E_i,E_l,E_j)U_j+\p_u\p_tv_i+\sum_{j=1}^n\p_tv_jX_{ji}&i=1,2,\cdots,n\\
X'_{ij}=\sum_{k,l=1}^nv_lR(E_i,E_j,E_l,E_k)U_k&i,j=1,2,\cdots,n\\
X_{ij}(u,0)=0&i,j=1,2,\cdots,n\\
U_i(u,0)=0&i=1,2,\cdots,n\\
U'_i(u,0)=\p_uv_i(u,0)&i=1,2,\cdots,n.
\end{array}\right.
\end{split}
\end{equation}
Here the symbol $'$ means taking derivative with respect to $t$.
\end{lem}
\begin{proof}Note that
\begin{equation}\label{eq-pt}
\pt:=\frac{\p\Phi}{\p t}=\sum_{i=1}^nv_iE_i
\end{equation}
and
\begin{equation}
\nabla_\pt\pt=\sum_{i=1}^n\p_tv_i(u,t)E_i(u,t).
\end{equation}
Then,
\begin{equation}\label{eq-pu-pt-pt}
\begin{split}
\nabla_{\pu}\nabla_\pt\pt=&\sum_{i=1}^n\p_u\p_tv_i(u,t)E_i(u,t)+\sum_{i=1}^n\p_tv_i(u,t)\nabla_{\pu}E_i(u,t)\\
=&\sum_{i=1}^n\left(\p_u\p_tv_i(u,t)+\sum_{j=1}^n\p_tv_jX_{ji}\right)E_i(u,t)
\end{split}
\end{equation}
So,
\begin{equation}
\begin{split}
\sum_{i=1}^nU_i''E_i=&\nabla_{\pt}\nabla_{\pt}\pu\\
=&\nabla_{\pt}\nabla_{\pu}{\pt}\\
=&\nabla_{\pu}\nabla_{\pt}{\pt}+R\lf(\pt,\pu\ri)\pt\\
=&\sum_{i=1}^n\left(\sum_{j,k,l=1}^nv_kv_lR(E_k,E_i,E_l,E_j)U_j+\p_u\p_tv_i+\sum_{j=1}^n\p_tv_jX_{ji}\right)E_i.
\end{split}
\end{equation}
This gives us the first equation of \eqref{eq-g-Jacobi}.
Moreover,
\begin{equation}\label{eq-pt-pu-E}
\begin{split}
\sum_{j=1}^nX'_{ij}E_j=&\nabla_{\pt}\nabla_{\pu}E_i(u,t)\\
=&\nabla_{\pu}\nabla_{\pt}E_i(u,t)+R\lf(\pt,\pu\ri)E_i\\
=&\sum_{j=1}^n\left(\sum_{k,l=1}^nv_lR(E_i,E_j,E_l,E_k)U_k\right)E_j.
\end{split}
\end{equation}
This gives us the second equation in \eqref{eq-g-Jacobi}. Finally, note that $$\pu\bigg|_{t=0}=0,\ \nabla_{\pu}E_i\big|_{t=0}=0$$
and
\begin{equation}
\nabla_{\pt} \pu\bigg|_{t=0}=\lim_{t\to 0^+}\nabla_{\pu}\pt=\sum_{i=1}^n\lim_{t\to 0^+}\nabla_{\pu}(v_iE_i)=\sum_{i=1}^n\p_uv_i(u,0)e_i.
\end{equation}
So, $X_{ij}(u,0)=0$, $U_i(u,0)=0$ and $U'_i(u,0)=\p_uv_i(u,0)$ for $i,j=1,2,\cdots,n$. This completes the proof of the lemma.
\end{proof}

We are now ready to prove Theorem \ref{thm-CAH}.
\begin{proof}[Proof of Theorem \ref{thm-CAH}] For $x\in M$, let $\gamma_0,\gamma_1:[0,1]\to M$ be two interior smooth curves joining $p$ to $x$. Since $M$ is simply connected, there is a smooth map $\Phi:[0,1]\times [0,1]\to M$ such that
\begin{equation}
\left\{\begin{array}{ll}\Phi(0,t)=\gamma_0(t)&\mbox{for $t\in [0,1]$}\\
\Phi(1,t)=\gamma_1(t)&\mbox{for $t\in [0,1]$}\\
\Phi(u,0)=p&\mbox{for $u\in [0,1]$}\\
\Phi(u,1)=x&\mbox{for $u\in [0,1]$},
\end{array}\right.
\end{equation}
and $\gamma_u(t)=\Phi(u,t)$ is an interior curve for any $u\in [0,1]$. Let
\begin{equation}
v(u,t)=P_t^0(\gamma_u)(\gamma'_u(t)).
\end{equation}
Then $\gamma_u$ is the development of $v(u,\cdot)$. Let $e_1,e_2,\cdots,e_n$ be an orthonormal basis of $T_pM$ and $E_i(u,t)$ be the parallel extension of $e_i$ along $\gamma_u$. Suppose that
$$v(u,t)=\sum_{i=1}^nv_i(u,t)e_i.$$
and
\begin{equation}
\frac{\p \Phi}{\p u}=\sum_{i=1}^n U_iE_i.
\end{equation}
Let $\tilde e_i=\varphi(e_i)$ for $i=1,2,\cdots, n$. Then
\begin{equation}
\varphi(v(u,t))=\sum_{i=1}^nv_i(u,t)\tilde e_i.
\end{equation}
Let $\tilde \Phi(u,t)=\dev(\tilde p, \varphi(v(u,\cdot)))(t)$ and $\tilde E_i$ be the parallel extension of $\tilde e_i$ along $\tilde\Phi(u,\cdot)$. Suppose that
\begin{equation}
\frac{\p \tilde \Phi}{\p u}=\sum_{i=1}^n\tilde U_i\tilde E_i.
\end{equation}
Note that $$R_M(E_i,E_j,E_k,E_l)=R_{\tilde M}(\tilde E_i,\tilde E_j,\tilde E_k,\tilde E_l)$$ by assumption. So, by Lemma \ref{lem-g-Jacobi}, $U_i$'s and $\tilde U_i$'s satisfy the same Cauchy problem of ODEs. By uniqueness of solution for Cauchy problems, we know that
\begin{equation}
\tilde U_i=U_i
\end{equation}
for $i=1,2,\cdots,n$. In particular, $\tilde U_i(u,1)=U_i(u,1)=0$ for $i=1,2,\cdots,n$. So $\tilde\Phi(0,1)=\tilde\Phi(1,1)$. This implies that $f$ is well defined.

Moreover, note that $f_{*}(\frac{\p\Phi}{\p u})=\frac{\p\tilde\Phi}{\p u}$ since $f(\Phi(u,t))=\tilde\Phi(u,t)$. So
\begin{equation}
\left\|f_{*}\lf(\frac{\p\Phi}{\p u}\ri)\right\|=\left\|\frac{\p\tilde\Phi}{\p u}\right\|=\sqrt{\sum_{i=1}^n\tilde U_i^2}=\sqrt{\sum_{i=1}^nU_i^2}=\left\|\frac{\p\Phi}{\p u}\right\|.
\end{equation}
This means that $f$ is a local isometry. It is not hard to see that $f(p)=\tilde p$ and $f_{*p}=\varphi$. This completes the proof of the theorem.

\end{proof}

\section{de Rham decomposition}
In this section, we come to prove Theorem \ref{thm-main}. First, we have the following simple conclusion for products of Riemannian manifolds.
\begin{lem}\label{lem-1}
Let $(M_1,g_1)$ and $(M_2,g_2)$ be two Riemannian manifolds and $M=M_1\times M_2$ be the product Riemannian manifold. Let $p=(p_1,p_2)\in M$ and  $v_i:[0,T]\to T_{p_i}M_i$ for $i=1,2$. Suppose the developments of $v_1$ and $v_2$ exists. Then,\\
(1) for any $t_i\in [0,T]$ with $i=1,2$,
\begin{equation*}
\begin{split}
\dev(\dev(p,v_1)(t_1),v_2)(t_2)=&(\dev(p_1,v_1)(t_1),\dev(p_2,v_2)(t_2))\\
=&\dev(\dev(p,v_2)(t_2),v_1)(t_1),
\end{split}
\end{equation*}
and the parallel displacement along the closed curve:
    $$\dev(p,v_2)|_{t_2}^0\cdot\dev(\dev(p,v_2)(t_2),v_1)|_{t_1}^0\cdot\dev(\dev(p,v_1)(t_1),v_2)|_0^{t_2}\cdot\dev(p,v_1)|_0^{t_1}$$
is the identity map of $T_pM$;\\
(2) for any $t\in [0,T]$,
\begin{equation*}
\begin{split}
\dev\left(\dev(p,v_1)(t),v_2\right)(t)=\dev(p,v)(t)=\dev(\dev(p,v_2)(t),v_1)(t)\\
\end{split}
\end{equation*}
and the parallel displacement along the closed curves
    $$\dev(p,v)|^0_t\cdot\dev(\dev(p,v_1)(t),v_2)|_0^t\cdot\dev(p,v_1)|_0^t$$
and $$\dev(p,v)|^0_t\cdot\dev(\dev(p,v_2)(t),v_1)|_0^t\cdot\dev(p,v_2)|_0^t$$
are the identity map of $T_pM$, where $v=(v_1,v_2)\in T_pM$.
\end{lem}
The conclusion of the lemma is clearly true from the product structure. For simplicity, we will not give the details of the proof here. Next, we come to show that similar conclusions with that of Lemma \ref{lem-1} hold on Riemannian manifolds with two nontrivial parallel distributions that are orthogonal complements of each other.
\begin{lem}\label{lem-2}
Let $(M^n,g)$ be a complete Riemmanian manifold with boundary, and $T_1$ and $T_2$ be two nontrivial parallel distributions on $M$ that are orthogonal complements of each other with $T_1$ containing the normal vectors of $\p M$ when $\p M\neq\emptyset$. Let $p$ be an interior point of $M$ and $v^i(t):[0,1]\to T_i(p)$ be a smooth curve for $i=1,2$. Let $v=v^1+v^2$. Then,\\
(1) the development of $v^2$ exists and stays in the interior of $M$;\\
(2) if $\dev(p,v^1)$ exists and is an interior curve, then so is $\dev(p,v)$ and
\begin{equation}\label{eq-parallelogram}
\begin{split}
\dev(\dev(p,v^1)(t),v^2)(t)=\dev(p,v)(t)=\dev(\dev(p,v^2)(t),v^1)(t)\\
\end{split}
\end{equation}
for any $t\in [0,1]$. Moreover the parallel displacements along the closed curves $$\dev(p,v)|^0_t\cdot\dev(\dev(p,v^1)(t),v^2)|_0^t\cdot\dev(p,v^1)|_0^t$$
and $$\dev(p,v)|^0_t\cdot\dev(\dev(p,v^2)(t),v_1)|_0^t\cdot\dev(p,v^2)|_0^t$$
are the identity map, for any $t\in [0,1]$.\\
(3) if $\dev(p,v)$ exists and is an interior curve, then so is $\dev(p,v^1)$.
\end{lem}
\begin{proof}
 (1) Let $I$ be the maximal interval that $\dev(p,v^2)$ exists. By completeness of $M$, it is clear that $I$ is closed and $\dev(p,v^2)(b)\in \p M$ with $I=[0,b]$ when $b<1$. This implies that $p$ is contained in the leaf of the foliation $T_2$ passing through $\dev(p,v^2)(b)$. However, because $T_2$ is orthogonal to normal vectors of $\p M$, we know that the leaf of $T_2$ passing through $\dev(p,v^2)(b)\in \p M$ must be contained in $\p M$. This contradicts that $p$ is an interior point. For the same reason, $\dev(p,v^2)(t)\in M\setminus \p M$ for any $t\in [0,1]$.\\
(2) Let $b>0$ be such that the development $\dev(p,v)$ of $v$ exists on $[0,b]$, and  $\dev(p,v)(t)$ is in the interior of $M$ for $t\in [0,b)$.  We first show:\\
\textbf{Claim 1.} The statement (2) is true for $t\in [0,b]$.\\
\emph{Proof of Claim 1.} Note that for any interior point $x\in M$, there is an open neighborhood $U$ of $x$ in $M$, such that $U=U_1\times U_2$ and each copy of $U_i$ is an integral submanifold of $T_i$ for $i=1,2$, we call $U$ a product neighborhood of $p$. Let $B_p(\delta)$ be contained in some product neighborhood. Then $\dev(p,v)(t)\in B_p(\delta)$
is contained in some product neighborhood for any $t<\frac\delta A$ where $A=\max_{t\in[0,1]}\|v(t)\|$. By Lemma \ref{lem-1}, the statement (2) is true for $t<\frac\delta A$.

Let $J=\{t\in [0,b]\ |\ \mbox{the statement (2) is  true up to }t.\}$  and let $t_0=\sup J$. By continuity, it is clear that $t_0\in J$. Suppose $t_0<b$.  By compactness, there is an $\epsilon>0$ such that for any $t\in [0,t_0]$, $B_{\dev(\dev(p,v^{1})(t_0),v^2)(t)}(\epsilon)$ is contained in some product neighborhood. Let $t_1\in [0,b]$ with $0<t_1-t_0<\frac\epsilon A$. We want to show that $t_1\in J$. This will be a contradiction. Then we are done in proving Claim 1.

Let $N$ be a natural number such that $\frac{t_0}N<\frac{\epsilon}{2A}$ and let $\xi_i=\frac{i t_0}N$ for $i=0,1,\cdots,N$. Note that $\dev(\dev(p,v)(t_0),v_{t_0})(t)$ for $t\in [0,t_1-t_0]$ is contained in $B_{\dev(p,v)(t_0)}(\epsilon)$ which is contained in  a product neighborhood. By Lemma \ref{lem-1} and \eqref{eq-dev-group}, we know that
\begin{equation}\label{eq-dev-1}
\begin{split}
&\dev(p,v)(t_1)\\
=&\dev(\dev(p,v)(t_0),v_{t_0})(t_1-t_0)\\
=&\dev(\dev(\dev(p,v)(t_0),v^{1}_{t_0})(t_1-t_0),v^{2}_{t_0})(t_1-t_0)\\
=&\dev(\dev(\dev(\dev(p,v^{1})(t_0),v^{2})(t_0),v^{1}_{t_0})(t_1-t_0),v^{2}_{t_0})(t_1-t_0).\\
\end{split}
\end{equation}
The last equality is by that $t_0\in J$.
We claim that
\begin{equation}\label{eq-dev-2}
\begin{split}
&\dev(\dev(\dev(p,v^{1})(t_0),v^{2})(t_0),v^{1}_{t_0})(t_1-t_0)\\
=&\dev(\dev(p,v^{1})(t_1),v^{2})(t_0).
\end{split}
\end{equation}
In fact, we will show that
\begin{equation}\label{eq-dev-3}
\begin{split}
&\dev(\dev(\dev(p,v^{1})(t_0),v^{2})(\xi_i),v^{1}_{t_0})(t_1-t_0)\\
=&\dev(\dev(p,v^{1})(t_1),v^{2})(\xi_i)
\end{split}
\end{equation}
for $i=0,1,\cdots, N$ inductively. The equality \eqref{eq-dev-2} is just \eqref{eq-dev-3} with $i=N$.

First, \eqref{eq-dev-3} is clearly true for $i=0$ by \eqref{eq-dev-group}. Suppose that \eqref{eq-dev-3} is true for some $i$ less than $N$. Note that
\begin{equation}
\dev(\dev(\dev(p,v^1)(t_0),v^2)(\xi_i),v_{t_0}^1)(t)\in B_{\dev(\dev(p,v^1)(t_0),v^2)(\xi_i)}(\epsilon)
\end{equation}
and
\begin{equation}
\dev(\dev(\dev(p,v^1)(t_0),v^2)(\xi_i),v_{\xi_i}^2)(t)\in B_{\dev(\dev(p,v^1)(t_0),v^2)(\xi_i)}(\epsilon)
\end{equation}
for $t\in[0,t_1-t_0]$ and $t\in [0,t_0/N]$ respectively. By Lemma \ref{lem-1}, we know that
\begin{equation*}
\begin{split}
&\dev(\dev(\dev(\dev(p,v^1)(t_0),v^2)(\xi_i),v_{t_0}^1)(t_1-t_0),v_{\xi_i}^2)(\xi_{i+1}-\xi_i)\\
=&\dev(\dev(\dev(\dev(p,v^1)(t_0),v^2)(\xi_i),v_{\xi_i}^2)(\xi_{i+1}-\xi_i),v^1_{t_0})(t_1-t_0)
\end{split}
\end{equation*}
By this and that \eqref{eq-dev-3} is true for $i$, we know that \eqref{eq-dev-3} is true for $i+1$. This proves \eqref{eq-dev-2} (See Figure \ref{picture} for help of understanding).

\begin{figure}
  \includegraphics[width=8cm]{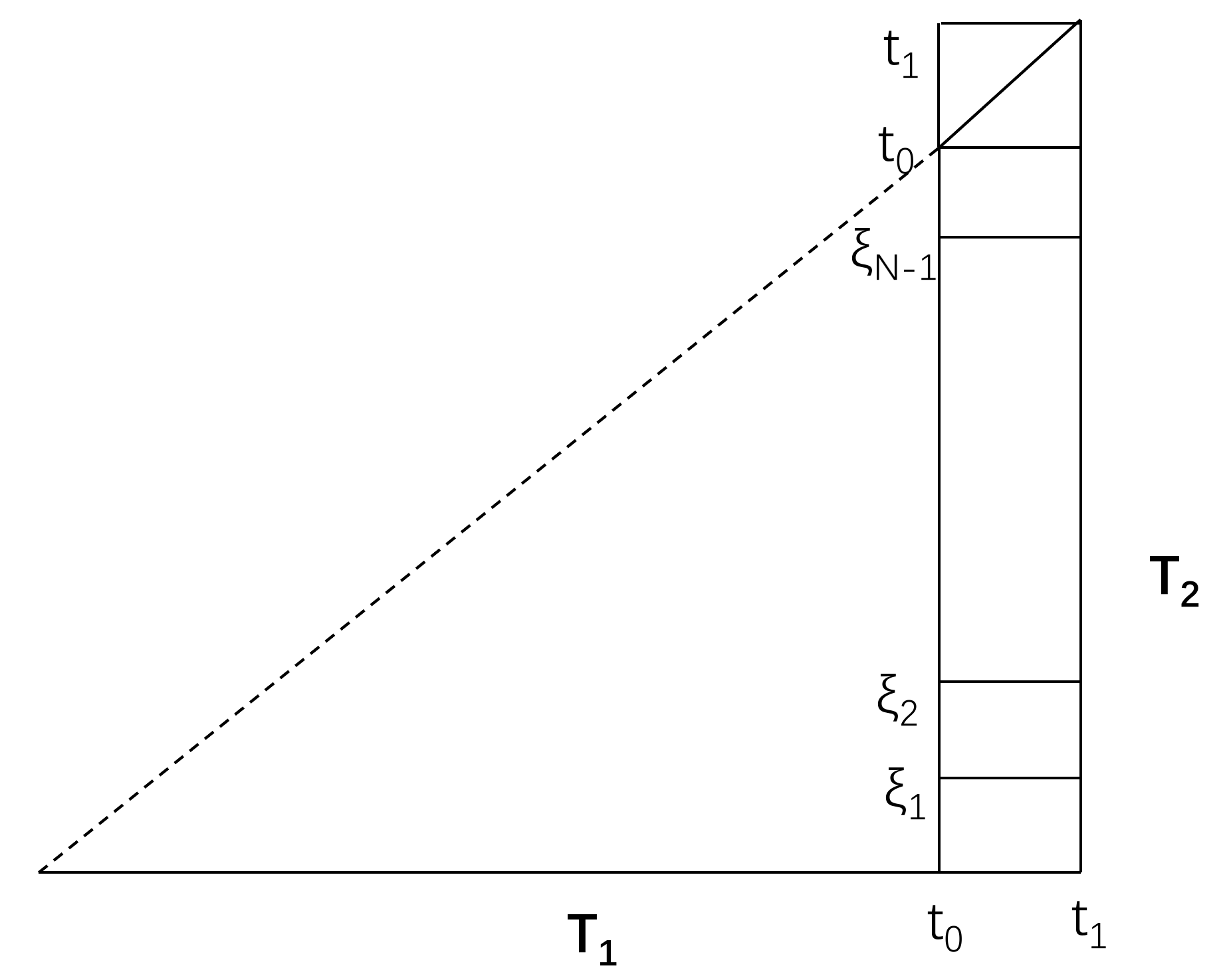}\\
  \caption{Each Small "rectangle" is contained in a product neighborhood so that Lemma \ref{lem-1} can be applied.}\label{picture}
\end{figure}

Substituting \eqref{eq-dev-2} into the last equality of \eqref{eq-dev-1} and using \eqref{eq-dev-group}, we know that
\begin{equation}
\dev(p,v)(t_1)=\dev(\dev(p,v^1)(t_1),v^2)(t_1).
\end{equation}
Similarly, one has
\begin{equation}
\dev(p,v)(t_1)=\dev(\dev(p,v^2)(t_1),v^1)(t_1).
\end{equation}
Moreover, by a similar argument, one can show that the parallel displacements along the two closed curves in the statement (2) is the identity map for $t=t_1$. This implies that $t_1\in J$ and we complete the proof of Claim 1.

We next come to show that the development of $v$ exists. Otherwise, by completeness of $(M,g)$, there is a $b\in (0,1)$ such that $\dev(p,v)$ exists on $[0,b]$, $\dev(p,v)(t)$ is in the interior of $M$ for $t\in [0,b)$ and  $\dev(p,v)(b)\in \p M$. By Claim 1, we know that $\dev(p,v)(b)$ can be joined to $\dev(p,v^1)(b)$ by the  curve $\dev(\dev(p,v^1)(b),v^2)(\cdot)$ which is tangential to $T_2$. This implies that $\dev(p,v^1)(b)\in \p M$. Because $\dev(p,v^1)$ is an interior curve, we know that $b=1$. This is a contradiction. By the same argument, we know that $\dev(p,v)$ is an interior curve. This completes the proof of (2).\\
(3) By the same argument as in the last paragraph of the proof for statement (2).
\end{proof}

Next, we have the following simple properties of curvature tensors for Riemannian manifolds with two nontrivial parallel distributions that are orthogonal complements of each other.
\begin{lem}\label{lem-curv}
Let $(M^n,g)$ be a Riemannian manifold, and $T_1$ and $T_2$ be two nontrivial parallel distributions that are orthogonal complements of each other on $M$. Then, \begin{enumerate}
\item for any $X,Y,Z,W\in T_pM$, suppose that $X=X_1+X_2$, $Y=Y_1+Y_2$, $Z=Z_1+Z_2$ and $W=W_1+W_2$ with $X_1,Y_1,Z_1,W_1\in T_1(p) $ and $X_2,Y_2,Z_2,W_2\in T_2(p)$. Then,
$$R(X,Y,Z,W)=R(X_1,Y_1,Z_1,W_1)+R(X_2,Y_2,Z_2,W_2);$$
\item let $\gamma:[0,1]\to M$ be a curve in $M$ that is tangential to $T_2$. Then, for any $X_1,Y_1,Z_1,W_1\in T_1(\gamma(0))$,
    \begin{equation}
    R(X_1,Y_1,Z_1,W_1)=R(P_0^1(\gamma)X_1,P_0^1(\gamma)Y_1,P_0^1(\gamma)Y_1,P_0^1(\gamma)Z_1).
    \end{equation}
\end{enumerate}
\end{lem}
\begin{proof}
(1) Since $T_i$ is parallel, $\nabla_\xi\eta\in T_i$ for any vector field $\xi$ and any vector field $\eta$ in $T_i$ with $i=1,2$. So,
\begin{equation*}
\begin{split}
R(X_1,Y_2,Z,W)=\vv<\nabla_Z\nabla_WX_1-\nabla_Z\nabla_WX_1-\nabla_{[Z,W]}X_1,Y_2>=0
\end{split}
\end{equation*}
since $\nabla_Z\nabla_WX_1-\nabla_Z\nabla_WX_1-\nabla_{[Z,W]}X_1\in T_1$. Similarly,
 \begin{equation*}
\begin{split}
R(X,Y,Z_1,W_2)=0.
\end{split}
\end{equation*}
This  gives us (1).\\
(2) Let $X_1(t),Y_1(t),Z_1(t),W_1(t)$ be the parallel extension of $X_1,Y_1,Z_1,W_1$ along $\gamma$ respectively. Because $T_1$ is parallel, $X_1(t),Y_1(t),Z_1(t),W_1(t)\in T_1$. So, by the second Bianchi identity and (1),
    \begin{equation*}
    \begin{split}
    &\frac{d}{dt}R(X_1(t),Y_1(t),Z_1(t),W_1(t))\\
    =&(\nabla_{\gamma'(t)}R)(X_1(t),Y_1(t),Z_1(t),W_1(t))\\
    =&-(\nabla_{W_1(t)}R)(X_1(t),Y_1(t),\gamma'(t),Z_1(t))-(\nabla_{Z_1(t)}R)(X_1(t),Y_1(t),W_1(t),\gamma'(t))\\
    =&0
    \end{split}
    \end{equation*}
    since $\gamma'\in T_2$. This gives us (2).

\end{proof}
We are now ready to prove Theorem \ref{thm-main}.
\begin{proof}[Proof of Theorem \ref{thm-main}] By (1) of Lemma \ref{lem-2}, we know that $M_2$ is a manifold without boundary.  On the other hand, because $T_1$ is transversal to $\p M$, we know that $M_1$ is a manifold with boundary.

Let $\gamma:[0,1]\to M_1\times M_2$ be an interior curve with $\gamma(0)=(p_1,p_2)$. Suppose that $\gamma(t)=(\gamma_1(t),\gamma_2(t))$. Let
$v^i(t)=P_t^0(\gamma_i)(\gamma'_i(t))$ for any $t\in [0,1]$ and $i=1,2$, and $v=(v^1,v^2)$. Then, $\gamma$ is the development of $v$. Let $\tilde v^i=(\iota_i)_{*p_i}(v^i)$ for $i=1,2,$ and $\tilde v=\tilde v^1+\tilde v^2=((\iota_1)_{*p_1}+(\iota_2)_{*p_2})v$. It is clear that $\tilde \gamma_i=\iota_i\circ\gamma_i$ is the developments of $\tilde v^i$ because $M_i$ is totally geodesic for $i=1,2$ . By (2) of Lemma \ref{lem-2}, we know that the development $\tilde\gamma$ of $\tilde v$ exists.

Let $X_i(t)$ be parallel vector fields along $\gamma$ for $i=1,2,3,4$. Suppose that
\begin{equation}
X_i(t)=(X_i^1(t),X_i^2(t))
\end{equation}
for $i=1,2,3,4$. Then, it is clear that $X_i^j$ is parallel along $\gamma_j$ for $i=1,2,3,4$ and $j=1,2$. Let $\tilde X_i^j(t)=(\iota_j)_{*\gamma_j(t)}X_i^j(t)$
for $i=1,2,3,4$ and $j=1,2$. By that $M_i$ is totally geodesic again, we know that $\tilde X^j_i$ is parallel along $\tilde \gamma_j$.

Let $\tilde X_i(0)=((\iota_1)_{*p_1}+(\iota_2)_{*p_2})(X_i(0))=\tilde X_i^1(0)+\tilde X_i^2(0)$ and $\tilde X_i(t)$ be the parallel extension of $\tilde X_i(0)$ along $\tilde \gamma$. By (2)  of Lemma \ref{lem-2}, we know that
\begin{equation}
\tilde X_i(1)=P_0^1(\sigma_2)(\tilde X_i^1(1))+P_0^1(\sigma_1)(\tilde X_i^2(1)).
\end{equation}
Here $\sigma_1(t)=\dev(\dev(p,v^2)(1),v^1)(t)$ which is tangential to $T_1$ and $\sigma_2=\dev(\dev(p,v^1)(1),v^2)(t)$ which is tangential to $T_2$. Then, by Lemma \ref{lem-curv}, we have
\begin{equation}
\begin{split}
&R_M(\tilde X_1(1),\tilde X_2(1),\tilde X_3(1),\tilde X_4(1))\\
=&R_M(\tilde X_1^1(1),\tilde X_2^1(1),\tilde X_3^1(1),\tilde X_4^1(1))+R_M(\tilde X_1^2(1),\tilde X_2^2(1),\tilde X_3^2(1),\tilde X_4^2(1))\\
=&R_{M_1}(X_1^1(1),X_2^1(1),X_3^1(1),X_4^1(1))+R_{M_2}( X_1^2(1), X_2^2(1), X_3^2(1),X_4^2(1))\\
=&R_{M_1\times M_2}(X_1(1),X_2(1),X_3(1),X_4(1)).
\end{split}
\end{equation}
Hence, by Theorem \ref{thm-CAH}, there is a local isometry $f:M_1\times M_2\to M$ such that $f(p_1,p_2)=p$ and $f_{*(p_1,p_2)}={\iota_1}_{*p_1}+{\iota_2}_{*p_2}$.

Conversely, for each interior curve $\tilde \gamma:[0,1]\to M$ in $M$, let
\begin{equation}
\tilde v(t)=P_t^0(\tilde \gamma)(\tilde \gamma'(t))
\end{equation}
for $t\in [0,1]$. Suppose that $\tilde v=\tilde v^1+\tilde v^2$ with $\tilde v^i\in T_i(p)$ for $i=1,2$. By Lemma \ref{lem-2}, we know that the developments $\tilde \gamma_i$ of $\tilde v_i$ exists for $i=1,2$. Because $M_i$  is the leaf of the foliation $T_i$ passing through $p$, there is a unique curve $\gamma_i:[0,1]\to M_i$ such that $\gamma_i(0)=p_i$ and $(\iota_i)_{*p_i}(\gamma'_i(t))=\tilde\gamma'_i(t)$ for $i=1,2$. Because $M_i$ is totally geodesic in $M$, $\gamma_i$ is the development of $v_i$ with ${\iota_i}_{*p_i}(v^i)=\tilde v^i$ for $i=1,2$. Let $\gamma=(\gamma_1,\gamma_2):[0,1]\to M_1\times M_2$. Then, $\gamma$ is the development of $v=(v^1,v^2)=((\iota_1)_{*p_1}+(\iota_2)_{*p_2})^{-1}(\tilde v)$. By the argument as before using Lemma \ref{lem-2} and Lemma \ref{lem-curv}, one can show that
\begin{equation}
R_{M}=\tau_{\gamma}^*R_{M_1\times M_2}.
\end{equation}
Hence, by Theorem \ref{thm-CAH}, there is a local isometry $h:M\to M_1\times M_2$ such that $h(p)=(p_1,p_2)$ and $h_{*p}=((\iota_1)_{*p_1}+(\iota_2)_{*p_2})^{-1}$.
Then, $f\circ h:M\to M$ is a local isometry with $f\circ h(p)=p$ and $(f\circ h)_{*p}=\id$. This implies that $f\circ h=\id$ and similarly $h\circ f=\id$. So $f$ is in fact an isometry. This completes the proof of the theorem.

\end{proof}

\end{document}